\documentclass[a4paper,11pt]{amsart}
\usepackage[latin1]{inputenc}
\usepackage[english]{babel}
\usepackage[T1]{fontenc}
\usepackage{amsmath,amssymb,amsfonts}

\usepackage{epsfig}
\usepackage{graphicx}
\usepackage{color}

\usepackage[all]{xy}
\usepackage{syntonly}

\usepackage[dvips]{hyperref}

\newcommand{\Z}[0]{\mathbb{Z}}
\newcommand{\R}[0]{\mathbb{R}}
\newcommand{\C}[0]{\mathbb{C}}
\newcommand{\Q}[0]{\mathbb{Q}}

\newcommand{\vs}[0]{\vspace{3mm}}
\newcommand{\ph}[0]{\varphi}

\newcommand{\q}[1]{\mbox{\bfseries{\textit{#1}}}}

\newcommand{\too}[0]{\longrightarrow}
\newcommand{\sst}[0]{\subset}

\newcommand{\qbin}[2]{ \left[ \begin{array}{c} #1 \\ #2
    \end{array} \right] }

\newcommand{\imm}[1][r] {\ar@{^{(}->}[#1]}

\renewcommand{\ni}[0]{\noindent}
\newcommand{\wI}[0]{\widetilde{I}}
\newcommand{\wqbin}[2]{ \widetilde{\left[ \begin{array}{c} #1 \\ #2
    \end{array} \right] }}
\newcommand{\tss}[0]{\supset}
\newcommand{\pmu}[0]{{\pm 1}}
\newcommand{\CA}[0]{\overline{C}}
\providecommand{\br}[1]{G_{#1}}

\def\qed{\ifmmode $\Box$ \else{\unskip\nobreak\hfil
\penalty50\hskip1em\null\nobreak\hfil $\Box$
\parfillskip=0pt\finalhyphendemerits=0\endgraf}\fi}

\def\W{W}

\newtheorem{teo}{Theorem}[section]
\newtheorem{lem}[teo]{Lemma}
\newtheorem{cor}[teo]{Corollary}
\newtheorem{prop}[teo]{Proposition}

\theoremstyle{definition}
\newtheorem{df}[teo]{Definition}

\theoremstyle{remark}
\newtheorem{os}[teo]{Remark}

\numberwithin{equation}{section}

\begin{document}

\title{Cohomology of affine Artin groups  and
applications}

\author[F. Callegaro]{Filippo Callegaro}
\address{Scuola Normale Superiore\\ P.za dei Cavalieri, 7, Pisa, Italy}
\email{f.callegaro@sns.it}

\author[D. Moroni]{Davide Moroni}
\address{Dipartimento di Matematica ``G.Castelnuovo''\\
P.za A.~Moro, 2, Roma, Italy -and- ISTI-CNR\\ Via G.~Moruzzi, 3,  Pisa, Italy}
\email{davide.moroni@isti.cnr.it}

\author[M. Salvetti]{Mario Salvetti}
\address{Dipartimento di Matematica ``L.Tonelli''\\ Largo B.~Pontecorvo, 5, Pisa, Italy}
\email{salvetti@dm.unipi.it}
\thanks{The third author is partially supported by M.U.R.S.T. 40\%}



\date {May 2006}

\begin{abstract}
The result of this paper is the determination of the cohomology of
Artin groups of type $A_n,\ B_n$ and $\tilde{A}_{n}$ with
non-trivial local coefficients. The main result (Theorem
\ref{t:cohomqt}) is an explicit computation of the cohomology of
the Artin group of type $B_n$  with coefficients over the module
$\Q[q^{\pmu},t^{\pmu}].$ Here the first $(n-1)$ standard
generators of the group act by $(-q)-$multiplication, while the
last one acts by $(-t)-$multiplication. The proof uses some
technical results from previous papers plus computations over a
suitable spectral sequence. The remaining cases follow from an
application of Shapiro's lemma, by considering some well-known
inclusions: we obtain the rational cohomology of the Artin group
of affine type $\tilde{A}_{n}$ as well as the cohomology of the
classical braid group $\mathrm{Br}_{n}$ with coefficients in the
$n$-dimensional representation presented in \cite{tong}. The
topological counterpart is the explicit construction of 
finite $CW-$complexes endowed with a free action of the Artin
groups, which are known to be $K(\pi,1)$ spaces in some cases (including
finite type groups). Particularly simple formulas for the 
Euler-characteristic of these orbit spaces are derived.

\end{abstract}

\maketitle

\vspace{1cm}

\section{Introduction}  Recall that for each Coxeter group $W$ one has
a group extension $G_W,$ usually called {\it Artin group} of type
$W$ (see Section \ref{sec:prelim}). In this paper we give a detailed
calculation of the cohomology of some Artin groups with non-trivial
local coefficients.  Let $R:=\Q[q^{\pmu},t^{\pmu}]$ be the ring of
two-parameters Laurent polynomials. The main result (Theorem
\ref{t:cohomqt}) is the computation of the cohomology of the Artin
group $G_{B_n}$ (of type $B_n$) with coefficients in the module
$R_{q,t}.$ The latter is the ring $R$ with the module structure
defined as follows: the generators associated to the first $n-1$
nodes of the Dynkin diagram of $B_n$ act by $(-q)-$multiplication;
the one associated to the last node acts by $(-t)-$multiplication.

Let $\ph_m(q)$ be the $m$-th cyclotomic polynomial in the variable
$q$. Define the $R$-modules ($m>1, \ i\geq 0$)
$$
\{m \}_i = R/(\ph_m(q), q^it+1).
$$
and for $m = 1$ set:
$$
\{ 1 \}_i = R/(q^it +1).
$$
Notice that the modules $\{ m \}_i$ are all non isomorphic as
$R$-modules. $\{m \}_i$ and $\{m'\}_{i'}$ are isomorphic as
$\Q[q^\pmu]$-modules if and only if $m = m'$ and are isomorphic as
$\Q[t^\pmu]$-modules if and only if $\phi(m) = \phi(m')$  ($\phi$ is
the Euler function) and $\frac{m}{(m,i)} = \frac{m'}{(m,i')}$.

Our main result is the following

\begin{teo} \label{t:cohomqt}
$$
H^i(G_{B_n}, R_{q,t})= \left\{
\begin{array}{ll}
\bigoplus_{d \mid n,\, 0 \leq k \leq d-2}
\{ d \}_k \oplus \{ 1 \}_{n-1} & \mbox{ if } i = n\\
\bigoplus_{
d \mid n, \, 0 \leq k \leq d-2, \, d \leq \frac{n}{j+1} }
\{ d \}_k & \mbox{ if }i = n -2j \\
\bigoplus_{d \nmid n, \, d \leq \frac{n}{j+1} } \{ d \}_{n-1} &
\mbox{ if }i = n -2j -1.
\end{array}
\right.
$$
\qed
\end{teo}

\ni The proof uses the spectral sequence associated to  a natural
filtration of the algebraic complex exhibited in \cite{S2}, plus
some technical results from \cite{DPS}.

We apply Shapiro's lemma to a well known inclusion of \
$G_{\tilde{A}_{n-1}}$ into $ G_{B_n}$  to derive the cohomology of
$G_{\tilde{A}_{n-1}}$ over the module $\Q[q^{\pmu}],$ the action of
each standard generator being $(-q)-$multiplication.

By considering another natural   inclusion of  \ $G_{{B}_{n}}$
into the classical braid group $\mathrm{Br}_{n+1}:=G_{A_n},$ we
also use Shapiro's lemma in order to identify the cohomology of
$G_{B_n}$ with coefficients in $R_{q,t}$ with that of
$\mathrm{Br}_{n+1}$ with coefficients in the irreducible
$(n+1)-$dimensional representation of $\mathrm{Br}_{n+1}$ found in
\cite{tong}, twisted by an abelian representation. We derive the
trivial $\Q-$cohomology of $G_{\tilde{A}_{n-1}}$ as well as the
cohomology of the braid group over the irreducible representation
in \cite{tong}.

Computation of the cohomology of Artin groups was done by several
people: for classical \emph{braid groups} and trivial coefficients
it was first given by F. Cohen \cite{coh}, and independently by A.
Va{\u\i}n\-{\v{s}}te{\u\i}n \cite{Ve} (see also
\cite{Ar,Br,Br-Sa,Fu}). For Artin groups of type $C_n,$\ $D_n$  see
\cite{Go}, while for the exceptional cases see \cite{S2}, where the
$\Z$-module structure was given, while the ring structure was
computed in \cite{La}. The case of non-trivial coefficients over the
module of Laurent polynomials $\mathbb Q[q^{\pm 1}]$ is interesting
because of its relation with the trivial $\Q$-cohomology of the
Milnor fibre of the naturally associated bundle. For the classical
braid groups see \cite{Fr,Ma,DPS}, while for cases $C_n,\ D_n$ see
\cite{DPS2}. For computations over the integral Laurent polynomials
$\mathbb Z[q^{\pm 1}],$ see \cite{CS} for the exceptional cases and
recently \cite{C2} for the case of braid groups, and \cite{DSS} for
the top cohomologies in all cases. In the case of Artin groups of
non-finite type, some computations were given in \cite{SS} and
\cite{CD}.

The computations of Theorem \ref{t:cohomqt} could be partially
extended to integral coefficients; however, major complications
occur because the Laurent polynomial ring $\mathbb Z[q^{\pm 1}]$ is
not a P.I.D..

In the last part we also indicate (see \cite{CMS}) an explicit 
construction of finite
CW-complexes which are retracts of {\it orbit spaces} associated to
Artin groups. The construction works as in \cite{S2}, with few
variations necessary for infinite type Artin groups (see also
\cite{CD} for a different construction). The Artin group identifies
with the fundamental group of the orbit space, and the standard
presentation follows easily (see \cite{Br, Ng, van}). The Euler
characteristic of the orbit space reduces to that of a simplicial
complex and in some cases one has a particularly simple
formula. It is conjectured that such orbit spaces are always
$K(\pi,1)$ spaces; for the affine groups, this is known in case
$\tilde{A}_n,\ \tilde{C}_n$ (see \cite{Ok,CP}) and recently also for
$\tilde{B}_n$ (\cite{CMS1}) (see also \cite{CD} for a different
class of Artin groups of infinite type).

Notice also the geometrical meaning of the two-parameters cohomology
of $G_{{B}_{n}}:$  similar to the one-parameter case, it is
equivalent to the trivial cohomology of the ``homotopy-Milnor
fibre'' associated to the natural map of the orbit space onto a
two-dimensional torus.

The main results of this paper were announced (without proof) in
\cite{CMS}.

\section{Preliminary results}\label{sec:prelim}
In this section we briefly fix the notation and recall some
preliminary results.
\subsection{Coxeter groups and Artin  groups}
A \emph{Coxeter graph} is a finite undirected graph, whose edges
are labelled with integers $\geq 3$ or with the symbol
 $\infty$.

Let  $S$ be the vertex set of a Coxeter graph. For every pair of
vertices $s, t \in S$ ($s\neq t$) joined by an edge, define
$m(s,t)$ to be the label of the edge joining them. If $s, \, t$
are not joined by an edge, set by convention $m(s,t)=2$. Let also
$m(s,s)=1$ (see \cite{bour,hum}).

Two groups are associated to  a Coxeter graph: the \emph{Coxeter
group} $W$ defined by
$$
W=\langle s \in S \, |\, (st)^{m(s,t)}=1 \,\, \forall s,t \in S
\,\, \textrm{such that}\,\, m(s,t)\neq \infty\rangle
$$
and the \emph{Artin  group} $G$ defined by (see \cite{Br-Sa}):
$$
G=\langle s \in S \, | \,
\begin{underbrace}{stst\ldots}\end{underbrace}_{m(s,t)-\mathrm{terms}}=\begin{underbrace}{tsts\ldots}\end{underbrace}_{m(s,t)-\mathrm{terms}}
\,\, \forall s,t \in S \,\, \textrm{such that}\,\, m(s,t)\neq
\infty\rangle.
$$
Loosely speaking, $G$ is the group obtained by dropping the
relations $s^2=1$ ($s\in S$) in the presentation for $W$.

In this paper, we are primarily interested in Artin  groups
associated to Coxeter graph of type $A_n$, $B_n$ and
$\tilde{A}_{n-1}$ (see Figure \ref{fig:dinkyn}).
\begin{figure}[hbtp]
\begin{center}
\begin{picture}(0,0)%
\includegraphics{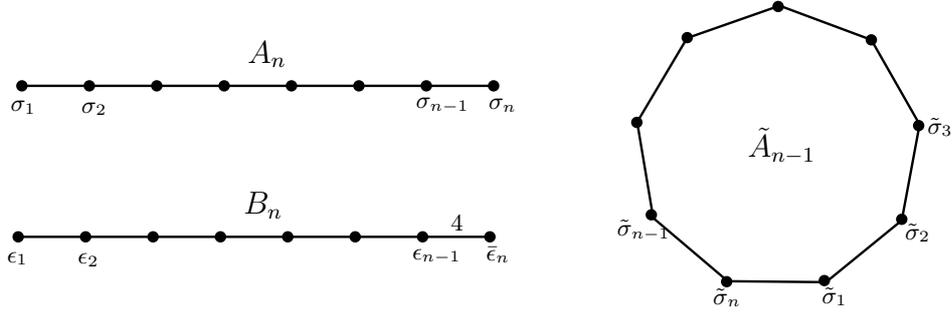}%
\end{picture}%
\setlength{\unitlength}{3947sp}%
\begingroup\makeatletter\ifx\SetFigFont\undefined%
\gdef\SetFigFont#1#2#3#4#5{%
  \reset@font\fontsize{#1}{#2pt}%
  \fontfamily{#3}\fontseries{#4}\fontshape{#5}%
  \selectfont}%
\fi\endgroup%
\begin{picture}(5720,1937)(-4070,-4328)
\put(-1076,-3076){\makebox(0,0)[lb]{\smash{{\SetFigFont{9}{10.8}{\familydefault}{\mddefault}{\updefault}{\color[rgb]{0,0,0}$\sigma_n$}%
}}}}
\put(1510,-3891){\makebox(0,0)[lb]{\smash{{\SetFigFont{9}{10.8}{\familydefault}{\mddefault}{\updefault}{\color[rgb]{0,0,0}$\tilde{\sigma}_2$}%
}}}}
\put(1650,-3228){\makebox(0,0)[lb]{\smash{{\SetFigFont{9}{10.8}{\familydefault}{\mddefault}{\updefault}{\color[rgb]{0,0,0}$\tilde{\sigma}_3$}%
}}}}
\put(994,-4289){\makebox(0,0)[lb]{\smash{{\SetFigFont{9}{10.8}{\familydefault}{\mddefault}{\updefault}{\color[rgb]{0,0,0}$\tilde{\sigma}_1$}%
}}}}
\put(534,-3376){\makebox(0,0)[lb]{\smash{{\SetFigFont{12}{14.4}{\familydefault}{\mddefault}{\updefault}{\color[rgb]{0,0,0}$\tilde{A}_{n-1}$}%
}}}}
\put(-4070,-4036){\makebox(0,0)[lb]{\smash{{\SetFigFont{9}{10.8}{\familydefault}{\mddefault}{\updefault}{\color[rgb]{0,0,0}$\epsilon_1$}%
}}}}
\put(-3630,-4043){\makebox(0,0)[lb]{\smash{{\SetFigFont{9}{10.8}{\familydefault}{\mddefault}{\updefault}{\color[rgb]{0,0,0}$\epsilon_2$}%
}}}}
\put(-1313,-3834){\makebox(0,0)[lb]{\smash{{\SetFigFont{9}{10.8}{\familydefault}{\mddefault}{\updefault}{\color[rgb]{0,0,0}$4$}%
}}}}
\put(-2598,-3729){\makebox(0,0)[lb]{\smash{{\SetFigFont{12}{14.4}{\familydefault}{\mddefault}{\updefault}{\color[rgb]{0,0,0}${B}_{n}$}%
}}}}
\put(-1552,-4015){\makebox(0,0)[lb]{\smash{{\SetFigFont{9}{10.8}{\familydefault}{\mddefault}{\updefault}{\color[rgb]{0,0,0}$\epsilon_{n-1}$}%
}}}}
\put(-4047,-3090){\makebox(0,0)[lb]{\smash{{\SetFigFont{9}{10.8}{\familydefault}{\mddefault}{\updefault}{\color[rgb]{0,0,0}$\sigma_1$}%
}}}}
\put(-3607,-3097){\makebox(0,0)[lb]{\smash{{\SetFigFont{9}{10.8}{\familydefault}{\mddefault}{\updefault}{\color[rgb]{0,0,0}$\sigma_2$}%
}}}}
\put(-1529,-3069){\makebox(0,0)[lb]{\smash{{\SetFigFont{9}{10.8}{\familydefault}{\mddefault}{\updefault}{\color[rgb]{0,0,0}$\sigma_{n-1}$}%
}}}}
\put(-2575,-2783){\makebox(0,0)[lb]{\smash{{\SetFigFont{12}{14.4}{\familydefault}{\mddefault}{\updefault}{\color[rgb]{0,0,0}${A}_{n}$}%
}}}}
\put(-278,-3870){\makebox(0,0)[lb]{\smash{{\SetFigFont{9}{10.8}{\familydefault}{\mddefault}{\updefault}{\color[rgb]{0,0,0}$\tilde{\sigma}_{n-1}$}%
}}}}
\put(317,-4297){\makebox(0,0)[lb]{\smash{{\SetFigFont{9}{10.8}{\familydefault}{\mddefault}{\updefault}{\color[rgb]{0,0,0}$\tilde{\sigma}_n$}%
}}}}
\put(-1092,-4022){\makebox(0,0)[lb]{\smash{{\SetFigFont{9}{10.8}{\familydefault}{\mddefault}{\updefault}{\color[rgb]{0,0,0}$\bar{\epsilon}_n$}%
}}}}
\end{picture}%

\end{center}
\caption{Coxeter graph of type  $A_n$, $B_n$ ($n\geq 2$) and
$\tilde{A}_{n-1}$ ($n\geq 3$). Labels equal to $3$, as usual, are
not shown. Moreover, to fix notation, every vertex is labelled
with the corresponding generator in the Artin  group.} \label{fig:dinkyn}
\end{figure}


\subsection{Inclusions of Artin groups}\label{sec:inclusions}
Let $\mathrm{Br}_{n+1}:=G_{A_n}$ be the braid group on $n+1$
strands and $\mathrm{Br}_{n+1}^{n+1}<\mathrm{Br}_{n+1}$ be the
subgroup of braids fixing the $(n+1)$-st strand. The group
$\mathrm{Br}_{n+1}^{n+1}$ is called the annular braid group.
Let $K_{n+1} = \{p_1, \ldots, p_{n+1} \}$ be a set of $n+1$ distinct points in $\C$.  The classical braid group $\mathrm{Br}_{n+1} = G_{A_n}$ can be realized as the fundamental group of the space of unordered configurations of $n+1$ points in $\C$ with basepoint $K_{n+1}$ (see the left part of Figure \ref{fig:cylinder}), with $K_6 = \{1, \ldots, 6 \}$). We can now think to the subgroup $\mathrm{Br}_{n+1}^{n+1}<\mathrm{Br}_{n+1}$ as the fundamental group of the space of unordered configurations of $n$ points in $\C^*$: in fact if we take $p_{n+1}= 0$ and $p_i \in S^1 \sst \C$ for $i \in 1, \ldots, n$, since for a braid $\beta \in \mathrm{Br}_{n+1}^{n+1}$ the orbit of the $(n+1)$-st point can be thought constant, up to homotopy, we can think to $\beta$ as a braid with $n$ strands in the anulus (see the right part of Figure \ref{fig:cylinder}).
\begin{figure}[hbtp]
 \begin{center}
\begin{picture}(0,0)%
\includegraphics{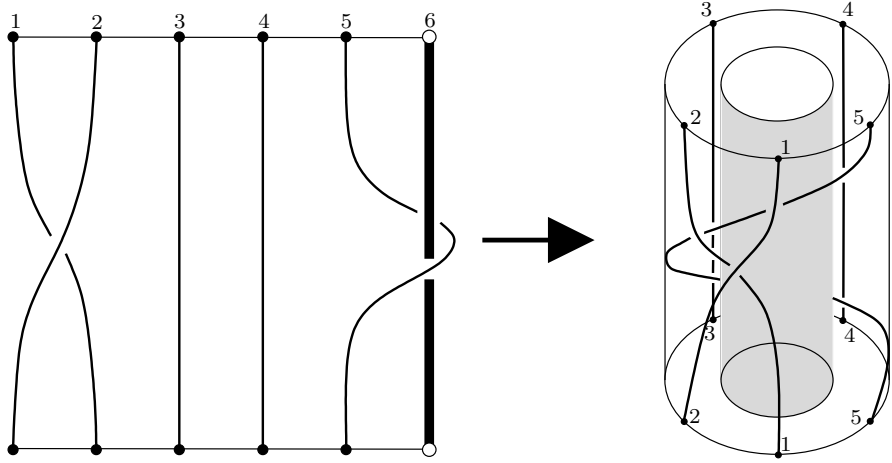}%
\end{picture}%
\setlength{\unitlength}{3947sp}%
\begingroup\makeatletter\ifx\SetFigFont\undefined%
\gdef\SetFigFont#1#2#3#4#5{%
  \reset@font\fontsize{#1}{#2pt}%
  \fontfamily{#3}\fontseries{#4}\fontshape{#5}%
  \selectfont}%
\fi\endgroup%
\begin{picture}(5501,2851)(-2605,-4336)
\put(2588,-1561){\makebox(0,0)[lb]{\smash{\SetFigFont{8}{9.6}{\familydefault}{\mddefault}{\updefault}{\color[rgb]{0,0,0}$4$}%
}}}
\put(1708,-1561){\makebox(0,0)[lb]{\smash{\SetFigFont{8}{9.6}{\familydefault}{\mddefault}{\updefault}{\color[rgb]{0,0,0}$3$}%
}}}
\put(2598,-3602){\makebox(0,0)[lb]{\smash{\SetFigFont{8}{9.6}{\familydefault}{\mddefault}{\updefault}{\color[rgb]{0,0,0}$4$}%
}}}
\put(1728,-3591){\makebox(0,0)[lb]{\smash{\SetFigFont{8}{9.6}{\familydefault}{\mddefault}{\updefault}{\color[rgb]{0,0,0}$3$}%
}}}
\put(-1566,-1635){\makebox(0,0)[lb]{\smash{\SetFigFont{8}{9.6}{\familydefault}{\mddefault}{\updefault}{\color[rgb]{0,0,0}$3$}%
}}}
\put(-1041,-1636){\makebox(0,0)[lb]{\smash{\SetFigFont{8}{9.6}{\familydefault}{\mddefault}{\updefault}{\color[rgb]{0,0,0}$4$}%
}}}
\put(1640,-2238){\makebox(0,0)[lb]{\smash{\SetFigFont{8}{9.6}{\familydefault}{\mddefault}{\updefault}{\color[rgb]{0,0,0}$2$}%
}}}
\put(2655,-2238){\makebox(0,0)[lb]{\smash{\SetFigFont{8}{9.6}{\familydefault}{\mddefault}{\updefault}{\color[rgb]{0,0,0}$5$}%
}}}
\put(2200,-2428){\makebox(0,0)[lb]{\smash{\SetFigFont{8}{9.6}{\familydefault}{\mddefault}{\updefault}{\color[rgb]{0,0,0}$1$}%
}}}
\put(1634,-4113){\makebox(0,0)[lb]{\smash{\SetFigFont{8}{9.6}{\familydefault}{\mddefault}{\updefault}{\color[rgb]{0,0,0}$2$}%
}}}
\put(2200,-4284){\makebox(0,0)[lb]{\smash{\SetFigFont{8}{9.6}{\familydefault}{\mddefault}{\updefault}{\color[rgb]{0,0,0}$1$}%
}}}
\put(2639,-4116){\makebox(0,0)[lb]{\smash{\SetFigFont{8}{9.6}{\familydefault}{\mddefault}{\updefault}{\color[rgb]{0,0,0}$5$}%
}}}
\put(-2583,-1637){\makebox(0,0)[lb]{\smash{\SetFigFont{8}{9.6}{\familydefault}{\mddefault}{\updefault}{\color[rgb]{0,0,0}$1$}%
}}}
\put(-2077,-1640){\makebox(0,0)[lb]{\smash{\SetFigFont{8}{9.6}{\familydefault}{\mddefault}{\updefault}{\color[rgb]{0,0,0}$2$}%
}}}
\put(-530,-1633){\makebox(0,0)[lb]{\smash{\SetFigFont{8}{9.6}{\familydefault}{\mddefault}{\updefault}{\color[rgb]{0,0,0}$5$}%
}}}
\put(-10,-1628){\makebox(0,0)[lb]{\smash{\SetFigFont{8}{9.6}{\familydefault}{\mddefault}{\updefault}{\color[rgb]{0,0,0}$6$}%
}}}
\end{picture}
\end{center}
\caption{A braid in $\mathrm{Br}_{6}^{6}$ represented as an
annular braid on $5$ strands.} \label{fig:cylinder}
\end{figure}

It is well known that the annular braid group is  isomorphic to
the Artin  group $G_{B_n}$ of type $B_n$. For a proof of the
following Theorem see \cite{crisp} or \cite{lam}.
\begin{teo}
Let $\sigma_1, \ldots, \sigma_n$ and $\epsilon_1, \ldots, \epsilon_{n-1},\bar{\epsilon}_n$ be respectively the standard generators for
$G_{A_n}$ and  $G_{B_n}$. Then, the map
\begin{align*}
G_{B_n}&\to \mathrm{Br}_{n+1}^{n+1}<\mathrm{Br}_{n+1}\\
\epsilon_i& \mapsto  \sigma_i \qquad \textrm{for $ 1 \leq i \leq n-1$}\\
\bar{\epsilon}_n &\mapsto \sigma_n^2
\end{align*}
 is an isomorphism.\qed
\end{teo}
Using the suggestion given by the identification with the annular
braid group, a new interesting presentation for $G_{B_n}$ can be
worked out. Let $\tau=\bar{\epsilon}_n \epsilon_{n-1} \cdots
\epsilon_2 \epsilon_1$.
\begin{figure}[htbp]
 \begin{center}
\scalebox{.7}{
\begin{picture}(0,0)%
\includegraphics{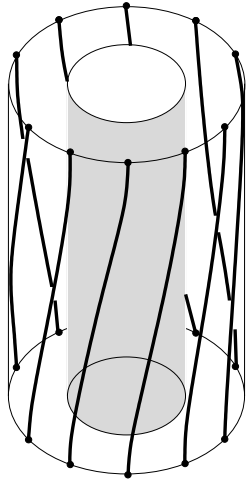}
\end{picture}%
\setlength{\unitlength}{3947sp}%
\begingroup\makeatletter\ifx\SetFigFont\undefined%
\gdef\SetFigFont#1#2#3#4#5{%
  \reset@font\fontsize{#1}{#2pt}%
  \fontfamily{#3}\fontseries{#4}\fontshape{#5}%
  \selectfont}%
\fi\endgroup%
\begin{picture}(2125,4269)(2086,-4576)
\end{picture}%
}
\end{center}
\caption{As an annular braid the element $\tau$ is obtained
turning the bottom annulus by a rotation of $2 \pi /n$.} \label{fig:tau_generator}
\end{figure}
It is easy to verify that:
$$
\tau^{-1} \epsilon_i \tau= \epsilon_{i+1} \quad \textrm{ for $1
\leq i < n-1$}
$$
i.e. conjugation by $\tau$ shifts forward the first $n-2$ standard
generators. By analogy, let  $\epsilon_n=\tau^{-1} \epsilon_{n-1}
\tau$.
We have  the following
\begin{teo}[\cite{peifer}] \label{teo:peifer}
The group $G_{B_n}$ has presentation $\langle \mathcal{G} |
\mathcal{R} \rangle$ where
\begin{align*}
\mathcal{G}=& \{ \tau, \epsilon_1, \epsilon_2, \ldots ,\epsilon_{n} \}\\
\mathcal{R}=& \{ \epsilon_i \epsilon_j=\epsilon_j \epsilon_i \quad
\textrm{for $i\neq j-1,j+1$}\} \cup\\
 {}& \{\epsilon_i \epsilon_{i+1} \epsilon_i=
\epsilon_{i+1} \epsilon_{i}
\epsilon_{i+1}\}\cup \\
{}& \{\tau^{-1} \epsilon_i \tau= \epsilon_{i+1}  \}
\end{align*}
where are all indexes should be considered modulo $n$. \qed
\end{teo}
Letting $\tilde{\sigma}_1, \tilde{\sigma}_2, \ldots,
\tilde{\sigma}_{n}$ be the standard generator of the Artin group
of type $\tilde{A}_{n-1}$, we have the following straightforward
corollary:

\begin{cor}[\cite{peifer}]\label{cor:peifer}
The map
$$G_{\tilde{A}_{n-1}} \owns \tilde{\sigma}_i \mapsto \epsilon_i\in G_{B_n}$$
gives an isomorphism between the group $G_{\tilde{A}_{n-1}}$ and the
subgroup of $G_{B_n}$ ge\-ne\-ra\-ted by $\epsilon_1, \ldots,
\epsilon_n$. Moreover, we have a semidirect product decomposition
$G_{B_n}\cong G_{\tilde{A}_{n-1}} \rtimes \langle \tau
\rangle$.\qed
\end{cor}

We have thus a ``curious'' inclusion of the Artin group of
infinite type $\tilde{A}_{n-1}$ into the Artin group of finite
type ${B_n}$.

\begin{os}
The proof of Theorem \ref{teo:peifer} presented in the original
paper is algebraic and based on Tietze moves; a somewhat more
coincise proof can however be obtained by standard topological
constructions. Indeed, one can exhibit an explicit infinite cyclic
covering $K(G_{\tilde{A}_{n-1}},1)\to K(G_{B_{n}},1)$ (see
\cite{all}).
\end{os}

\subsection{${\mathbf{(q,t)}}$-weighted Poincar\'e series for
${\mathbf{B_n}}$}

For future use in cohomology computations, we are interested in a
 $(q,t)$-analog of the usual Poincar\'e series for $B_n$, that is an analog of
 the Poincar\'e series with coefficients in the ring $R=\Q[q^\pmu, t^\pmu]$ of Laurent
 polynomials. This result
 and similar ones are studied in
 \cite{reiner}, to which we refer for details. We also use classical
 results from \cite{bour,hum} without further reference.

Consider the Coxeter group $W$ of type $B_n$ with its standard
generating
reflections $s_1, s_2, \ldots, s_n$.
For $w \in W$, let $n(w)$ be the number of times $s_n$ appears in
a reduced expression for $w$. By standard facts, $n(w)$ is
well-defined.\\
We define the $(q,t)$-weighted Poincar\'e
series for the Coxeter group of type $B_n$ as the sum $$\W(q,t)=\sum_{w \in W}
q^{\ell(w)-n(w)} t^{n(w)},$$ where $\ell$ is the length function.

We recall some notation. We define the $q$-analog of the number
$m$ by the polynomial
$$
[m ]_q := 1 + q + \cdots q^{m-1} =
\frac{q^m -1}{q-1}.
$$
 Notice that $[m ] = \prod_{i \mid
m, \, i\neq 1 } \ph_m(q)$, where we denote with $\ph_m(q)$ the
$m$-th cyclotomic polynomial in the variable $q$. Moreover we
define the $q$-factorial analog $[m]_q!$ as the product
$$\prod_{i=1}^{m}[i]_q$$ and the $q$-analog of the binomial $\binom{m}{i}$
as the polynomial
$$
\qbin{m}{i}_{q}: = \frac{[m]_q!}{[i]_q!
[m-i]_q!}.
$$
We can also define the $(q,t)$-analog of an even
number
$$
[2m]_{q,t} := [m]_q (1+tq^{m-1})
$$
and of the double factorial
$$
[2m]_{q,t}!! := \prod_{i=1}^m [2i]_{q,t}\ =\ [ m ]_q!
\prod_{i=0}^{m-1} (1+tq^i).
$$
Finally, we define the polynomial
$$
\qbin{m}{i}_{q,t}': = \frac{[2m]_{q,t}!!}{[2i]_{q,t}!! [m-i]_q!} \ =
\qbin{m}{i}_{q}\prod_{j=i}^{m-1}(1+tq^j).
$$

\begin{prop}[\cite{reiner}] \label{p:qtpoincare}
$$\W(q,t)=
[2n]_{q,t}!!.$$
\end{prop}
\begin{proof}
Consider the parabolic subgroup $W_I$ associated to the subset of
reflections $I=\{s_1, \ldots, s_{n-1}\}$. Notice that $W_I$ is
isomorphic to the symmetric group on $n$ letters $A_{n-1}$ and
that it has index $2^n$ in $B_n$. Let $W^I$ be the set of minimal
coset representatives for $W/W_I$. Then, by multiplicative
properties on reduced expressions:
\begin{align}
\W(q,t)=&\sum_{w \in W} q^{\ell(w)-n(w)} t^{n(w)}\nonumber\\
=& \Big( \sum_{w' \in W^I} q^{\ell(w')-n(w')} t^{n(w')}\Big) \cdot
\Big( \sum_{w'' \in W_I} q^{\ell(w'')-n(w'')} t^{n(w'')}\Big).
\label{formula:1}
\end{align}
Clearly, for elements $w''\in W_I$, we have $n(w'')=0$; so the
second factor in (\ref{formula:1}) reduces to the well-known
Poincar\'e series for $A_{n-1}$:
$$
 \sum_{w'' \in W_I} q^{\ell(w'')-n(w'')} t^{n(w'')}=[n]_q!.
$$
To deal with the first factor, instead, we explicitly enumerate
the elements of $W^I$. Let $p_i=s_i s_{i+1} \cdots s_n$ for $1
\leq i \leq n$. Then, it can be easily verified that
$W^I=\{p_{i_r}p_{i_{r-1}}\cdots p_{i_2}p_{i_1}\,|\, i_1 < i_2 <
\cdots < i_{r-1}<i_r \}$. Notice that $n(p_{i_r}p_{i_{r-1}}\cdots
p_{i_2}p_{i_1})=r$ and $\ell(p_{i_r}p_{i_{r-1}}\cdots
p_{i_2}p_{i_1})=\sum_{j=1}^{r} \ell(p_{i_j})=
\sum_{j=1}^{r}(n+1-i_j)$. Thus,
$$
\sum_{w' \in W^I} q^{\ell(w')-n(w')}
t^{n(w')}=\prod_{i=0}^{n-1}(1+tq^i).
$$
Finally,
$$
\W(q,t)= \Big(\prod_{i=0}^{n-1}(1+tq^i)\Big) [n]_q!=[2n]_{q,t}!!.
$$
\end{proof}

\section{The cohomology of $G_{B_n}$} \label{s:cohomBn}

\subsection{Proof of the Main Theorem}

In this Section we prove Theorem \ref{t:cohomqt} enunciated in the
introduction. We use the notations given in the Introduction.

To perform our computation we will use the complex discovered in
\cite{S2}, \cite{DS} (notice: an equivalent complex was discovered
by different methods in \cite{Sq}), and the spectral sequence
induced by a natural filtration.

The complex that computes the cohomology of $G_{B_n}$ over
$R_{q,t}$ is given as follows (see \cite{S2}):
$$
C_n^* = \bigoplus_{\Gamma \sst I_n} R.\Gamma
$$
where $I_n$ denote the set $\{1, \ldots, n\}$ and the graduation
is given by $\mid \Gamma \mid$.

The set $I_n$ corresponds to the set of nodes of the Dynkin
diagram of $B_n$ and in particular the last element, $n$,
corresponds to the last node.

It is useful to consider also the complex $\CA_n^*$ for the
cohomology of $G_{A_n}$ on the local system $R_{q,t}$. In this
case the action associated to a standard generator is always the
$(-q)$-multiplication and so the complex $\CA_n^*$ and its
cohomology are free as $\Q[t^\pm]$-modules. The complex $\CA_n^*$
is isomorphic to $C_n^*$ as a $R$-module. In both complexes the
coboundary map is
\begin{equation} \label{e:bordo}
\delta (q,t) (\Gamma) = \sum_{j \in I_n \setminus \Gamma}
(-1)^{\sigma(j,\Gamma)} \frac{W_{\Gamma \cup \{ j \}}(q,t)} {W_{
\Gamma}(q,t)} (\Gamma \cup \{j\})
\end{equation}
where $\sigma(j, \Gamma)$ is the number of elements of $\Gamma$
that are less than $j$. In the case $A_n$ $ W_\Gamma (q,t)$ is the
Poincar\'e polynomial of the parabolic subgroup $W_\Gamma \sst
A_n$ generated by the elements in the set $\Gamma$, with weight
$-q$ for each standard generator, while in the case $B_n$ $
W_\Gamma (q,t)$ is the Poincar\'e polynomial of the parabolic
subgroup $W_\Gamma \sst B_n$ generated by the elements in the set
$\Gamma$, with weight $-q$ for the first $n-1$ generators and $-t$
for the last generator.

Using Proposition \ref{p:qtpoincare} we can give an explicit
computation of the coefficients $\frac{W_{\Gamma \cup \{ j
\}}(q,t)} {W_{ \Gamma}(q,t)}$. For any $ \Gamma \sst I_n$, let
$\overline{\Gamma}$ be the subgraph of the Dynkin diagram $B_n$
which is spanned by $\Gamma$. Recall that if $\overline{\Gamma}$
is a connected component of the Dynkin diagram of $B_n$ without
the last element, then $$W_{\Gamma}(q, t) = [ m+1 ]_q!,$$ where $m
= \mid \Gamma \mid$. If $\overline{\Gamma}$ is connected and
contains the last element of $B_n$, then
$$W_{\Gamma}(q, t) = [2m]_{q,t}!!, $$ where $m = \mid \Gamma \mid$.

If $\overline{\Gamma}$ is the union of several connected
components of the Dynkin diagram, $\overline{\Gamma} =
\overline{\Gamma}_1 \cup \cdots \cup \overline{\Gamma}_k$, then
$W_\Gamma(q,t)$ is the product $$\prod_{i=1}^k W_{\Gamma_i}(q,
t)$$ of the factors corresponding to the different components.

If $j \notin \Gamma$ we can write $\overline{\Gamma}(j)$ for the
connected component of $\overline{\Gamma \cup \{j\}}$ containing
$j$. Suppose that $m = \mid \Gamma(j) \mid$ and $i$ is the number
of elements in $\Gamma(j)$ greater than $j$. Then, if $n \in
\Gamma(j)$ we have
$$
\frac{W_{\Gamma \cup \{ j \}}(q,t)} {W_{ \Gamma}(q,t)} =
\qbin{m}{i}_{q,t}'
$$
and
$$
\frac{W_{\Gamma \cup \{ j \}}(q,t)} {W_{ \Gamma}(q,t)} =
\qbin{m+1}{i+1}_{q}
$$
otherwise.

It is convenient to represent generators $\Gamma \sst I_n$ by
their characteristic functions $I_n \to \{0,1\}$ so, simply by
strings of $0$s and $1$s of length $n$.

We define a decreasing filtration $F$ on the complex
$(C^*_n,\delta)$: $F^sC_n$ is the subcomplex generated by the
strings of type $A1^s$ (ending with a string of $s$ $1$s) and we
have the inclusions
$$
C_n = F^0C_n \tss F^1C_n \tss \cdots \tss F^nC_n = R.1^n \tss
F^{n+1}C_n = 0.
$$
We have the following isomorphism of complexes:
\begin{equation} \label{e:shift}
(F^sC_n/F^{s+1}C_n) \simeq \CA_{n-s-1} [s]
\end{equation}
where $\CA_{n-s-1}$ is the complex for $G_{A_{n-s-1}}$ and the
notation $[s]$ means that the degree is shifted by $s$. Let $E_*$
be the spectral sequence associated to the filtration $F$. The
equality \ref{e:shift} tells us how the $E_1$ term of the
spectral sequence looks like. In fact for $0 \leq s \leq n-2$ we
have
\begin{equation} \label{e:sseq1}
E_1^{s,r} = H^{r}(G_{A_{n-s-1}}, R_{q,t}) = H^{r}(G_{A_{n-s-1}},
\Q[q^\pmu]_q)[t^\pmu]
\end{equation}
since the $t$-action is trivial. For $s = n-1$ and $s = n$ the
only non trivial elements in the spectral sequence are
\begin{equation} \label{e:sseq2}
E_1^{n-1,0} = E_1^{n,0} = R.
\end{equation}

In order to prove Theorem \ref{t:cohomqt} we need to state the
following lemmas.

\begin{lem} \label{l:ideali2}
Let $I(n,k)$ be the ideal generated by the polynomials
$$
\qbin{n}{n-d}'_{q,t} \qquad\textrm{    for $d \mid n$ and $d \leq
k$}
$$
If $k \mid n$ the map
$$
\alpha_{n,k}:R/(\ph_k(q)) \to R/I(n,k-1)
$$
induced by the multiplication by $\qbin{n}{n-k}'_{q,t}$ is well
defined and is injective.
\end{lem}

\noindent {{\bf Remark. }} \ The fact that this map is well
defined will follow automatically from the general theory of
spectral sequences, as it is clear from the proof of Theorem
\ref{t:cohomqt}. However, below we prove it by other means.

\begin{proof}
Let $d,k$ be positive integers such that $d \mid n$ and $k \mid
n$. We can observe that $\ph_d(q) \mid \qbin{n}{k}_q =
\qbin{n}{n-k}_q $ if and only if $d \nmid k$. Moreover each factor
$\ph_d$ appears in $\qbin{n}{k}_q$ at most with exponent $1$.

Let $J(n,k)$ be the ideal generated by the polynomials
$\qbin{n}{n-d}_q$ for $d \mid n$ and $d \leq k$. It is easy to see
that we have the following inclusion:
$$
\prod_{i= n-k}^{n-1} (1+tq^i) J(n,k) \sst I(n,k).
$$
Moreover $J(n,k)$ is a principal ideal and is generated by the
product
$$p_{n,k}(q) = \prod_{d \mid n, k < d} \ph_d(q).$$
It follows that $\qbin{n}{n-k}_q \ph_k(q) \in J(n,k-1)$ and so
$\qbin{n}{n-k}'_{q,t} \ph_k(q) \in I(n,k-1)$. This proves that the
map $\alpha_{n,k}$ is well defined.

Now we notice that the factor $\ph_k(q)$ divides each generator
of $I(n,k-1)$, but does not divide $\qbin{n}{n-k}'_{q,t}$. This
imply that $\alpha_{n,k}$ is not the zero map and that every
polynomial in the kernel of $\alpha_{n,k}$ must be a multiple of
$\ph_k(q)$, hence the map must be injective.
\end{proof}

\begin{lem} \label{l:ideali1}
Let $I(n)$ be the ideal generated by the polynomials
$$\qbin{n}{n-d}'_{q,t} \qquad \textrm{     for $d \mid n$}$$ Then $I(n)$ is the direct
product of the ideals $I_{i,d} = (\ph_d(q), q^it+1)$ for $d \mid
n$ and $0 \leq i \leq d-2$ and of the ideal
$I_{n-1}=(q^{n-1}t+1)$. Moreover the ideals $I_{i,d}$ and
$I_{n-1}$ are pairwise coprime.
\end{lem}

\begin{proof}[Proof: ]
Notice that the polynomial $(1+tq^{n-1})$ divides each generator of
the ideal $I(n)$, so we can write
$$
I(n) = (1+tq^{n-1})\wI(n)
$$
where $\wI(n)$ is the ideal generated by the polynomials
$$\wqbin{n}{n-d}'_{q,t} := \qbin{n}{n-d}'_{q,t}/(1+tq^{n-1})$$
Let $n= d_1 > \cdots > d_h = 1$ be the list of all the divisors of
$n$ in decreasing order. If we set
\begin{align*}
    P_i :=& \ph_{d_i}(q) \; \mbox{ and } \\
    Q_i :=&
\prod_{j = d_{i+1}+1}^{d_i} (1+tq^{n-j})
\end{align*}
 we can rewrite our ideal as
\begin{equation} \label{e:ideali1}
\begin{split}
    \wI(n) =& \left( \qbin{n}{n-d_h},\,
    \qbin{n}{n-d_{h-1}}Q_{h-1},\,
    \qbin{n}{n-d_{h-2}}Q_{h-2}Q_{h-1},\,
\ldots \right. \\
        {} & \left. \ldots, \qbin{n}{n-d_2}Q_2\cdots Q_{h-1}, \, Q_1 \cdots Q_{h-1}
\right) \end{split}
\end{equation}
We claim that we can reduce to the following set of generators:
\begin{equation} \label{e:ideali2}
\begin{split}
    \wI(n) =& \left( P_1 \cdots P_{h-1}, \, P_1 \cdots P_{h-2}
Q_{h-1},\, P_1\cdots P_{h-3}Q_{h-2}Q_{h-1} \ldots \right.\\
    & \left. \ldots, \, P_1Q_2 \cdots
Q_{h-1}, \, Q_1 \cdots Q_{h-1} \right)
\end{split}
\end{equation}
The first generator is the same in both equations and the $j$-th
generator in Equation \ref{e:ideali2} divides the corresponding
generator in Equation \ref{e:ideali1}. Now suppose that a factor
$\ph_m(q)$ divides $\qbin{n}{n-d_j}$ but does not divide $P_1
\cdots P_{j-1}$. We may distinguish two cases:
\begin{itemize}
    \item[i)] Suppose that $m \nmid n$.  Then we can get rid of the factor $\ph_m(q)$
in $\qbin{n}{n-d_j}$ with an opportune combination with the
polynomial $$P_1 \cdots P_{h-1}$$
    \item[ii)] Suppose $m \mid n$. Then $m = d_l$ for some
$l > j$ and we can get rid of $\ph_m(q)$ using a suitable
combination with the polynomial $$P_1 \cdots P_{l-1} Q_l \cdots
Q_{h-1}$$
\end{itemize}
We may now proceed inductively. Supposing we have already reduced
the first $j-1$ terms,  we can reduce the $j$-th term of the ideal
in Equation \ref{e:ideali1} to the corresponding term in
Equation \ref{e:ideali2}.

Now we observe that if $J, I_1, I_2$ are ideals and $I_1 + I_2 =
(1)$, then $(J, I_1I_2) = (J, I_1)(J, I_2)$. Since the polynomials
$P_i$ are all coprime, we can apply this fact to the ideal $\wI(n)$ $h-2$ times.  At the $i$-th step we set $$I_1
=(P_i),$$ $$I_2=(P_{i+1} \cdots P_{h-1}, P_{i+1} \cdots P_{h-2}
Q_{h-1}, \ldots, Q_{i+1} \cdots Q_{h-1}),$$ $$J=(Q_i\cdots
Q_{h-1}).$$ So we can factor $\wI(n)$ as
$$(P_1, Q_1\cdots Q_{h-1})(P_2 \cdots P_{h-1}, P_2\cdots P_{h-2}Q_{h-1}, Q_2\cdots Q_{h-1}) = \cdots $$
$$
\cdots = (P_1, Q_1 \cdots Q_{h-1})(P_2, Q_2 \cdots Q_{h-1})\cdots
(P_{h-1}, Q_{h-1}).
$$
Finally we can split $(P_{s}, Q_s \cdots Q_{h-1})$ as the product
$$
(P_s, 1+tq^{n-d_s})\cdots(P_s, 1+tq^{n-d_h-1}).
$$
So we have reduced the ideal $I(n)$ in the product stated in the
Lemma and it is easy to check that all the ideals of the splitting
are coprime.
\end{proof}
\begin{proof} [Proof of Theorem \ref{t:cohomqt}]
We can now prove our Theorem using the spectral sequence described
in the Equations \ref{e:sseq1} and \ref{e:sseq2}.

We introduce, as in \cite{DPS}, the following notation for the
generators of the spectral sequence:
\begin{eqnarray*}
w_h & = & 01^{h-2}0\\
z_h & = & 1^{h-1}0 + (-1)^h 01^{h-1}\\
b_h & = & 01^{h-2}\\
c_h & = & 1^{h-1}\\
z_h(i) & = & \sum_{j=0}^{i-1} (-1)^{hj}w_h^j z_h w_h^{i-j-1}\\
v_h(i) & = & \sum_{j=0}^{i-2} (-1)^{hj}w_h^j z_h w_h^{i-j-2} b_h +
(-1)^{h(i-1)} w_h^{i-1} c_h
\end{eqnarray*}

We write $\{m \}[t^\pmu]$ for the module $R/(\ph_m(q))$. The
$E_1$-term of the spectral sequence has a module $\{m \}[t^\pmu]$
in position $(s,r)$ if and only if one of the following condition
is satisfied:

a)$m \mid n-s-1$ and $ r = n-s-2 \frac{n-s-1}{m}$;

b)$m \mid n-s$ and $ r=n-s+1-2(\frac{n-s}{m})$.

Moreover we have modules $R$ in position $(n-1,0)$ and $(n,0)$. We
now look at the $d_1$ map between these two modules. Notice that $E_1^{n-1,0}$
is generated by the string $01^{n-1}$ and $E_1^{n,0}$ is generated
by the string $1^n$. Furthermore the map
$$
d_1^{n-1,0}: E_1^{n-1,0}\to E_1^{n,0}
$$
is given by the multiplication by $\qbin{n}{n-1}_{q,t}' =
[n]_q(1+tq^{n-1}) $ and is injective. It turns out that
$E_2^{n-1,0} = 0$ and $E_2^{n,0} = R/([n]_q(1+tq^{n-1}))$.
Moreover all the following terms $E_j^{n,0}$ are quotient of
$E_2^{n,0}$.

Notice that every map between modules of kind $\{m \}[t^\pmu]$ and
$\{ m'\} [t^\pmu]$ must be zero if $m \neq m'$. So we can study
our spectral sequence considering only maps between the same kind
of modules.

First let us consider an integer $m$ that doesn't divide $n$. Say
that $m \mid n+c$ with $1 \leq c < m$ and set $i = \frac{n+c}{m}$.
The modules of type $\{ m\} [t^\pmu]$ are:

$$
\begin{array}{ll}
E_1^{\lambda m-c-1,n+c-\lambda(m-2) -2i+1}
& \mbox{generated by } z_m(i-\lambda)01^{\lambda m-c-1}\\
E_1^{\lambda m-c ,n+c-\lambda(m-2) -2i+1} & \mbox{generated by }
v_m(i-\lambda)01^{\lambda m-c}
\end{array}
$$
for $\lambda = 1, \ldots, i -1$.

Here is a diagram for this case (we use the notation $\q{h}$ for $\{m \}[t^\pmu]$):
\begin{center}
\begin{tabular}{|l}
\xymatrix @R=1pc @C=1pc {
& & \q{h}\ar[r]^{d_1} & \q{h} \\
& &       &       & & & \cdots \ar[r]^{d_1} & \cdots \\
& &       &       & & &        &        & & & \q{h} \ar[r]^{d_1} & \q{h} \ar[rrrd]^{0}\\
& &       &       & & &        &        & & &                &       & & & R/I}\\
\hline
\end{tabular}
\end{center}
\vs

The map $$d_1:E_1^{\lambda m-c-1,n+c-\lambda(m-2)
-2i+1}\rightarrow E_1^{\lambda m-c, n+c-\lambda(m-2) -2i+1}$$ is
given by the multiplication by $\qbin{\lambda m - c}{\lambda m -c
-1}'_{q,t} = [\lambda m -c]_q(1+tq^{\lambda m -c -1}).$ Since
$\ph_m(q) \nmid [\lambda m -c]_q$ the map is injective and in the
$E_2$-term we have:

$$
\begin{array}{lcl}
E_2^{\lambda m-c-1,n+c-\lambda(m-2) -2i+1} & = & 0 \\
E_2^{\lambda m-c ,n+c-\lambda(m-2) -2i+1} & = & \{m\}_{\lambda m
-c -1} = \{m\}_{m-c -1}
\end{array}
$$
for $\lambda = 1, \ldots, i -1$.

The other map we have to consider is
$$d_m^{n-m, m-1}: E_m^{n-m,m-1} \to E_m^{n,0}.$$
The module $E_m^{n-m, m-1} = \{m\}_{m-c -1}$
is generated by $1^{m-1}01^{n-m}$ and so the map is the
multiplication by $\qbin{n}{n-m}'_{q,t}$. Since $(1+tq^{n-1})$
divides the coefficient $\qbin{n}{n-m}'_{q,t}$, the image of the
map $d_m^{n-m, m-1}$ must be contained in the submodule
$$(1+tq^{n-1})E_m^{n,0} = (1+tq^{n-1})R/([n]_q(1+tq^{n-1}))$$
that is in the quotient $R/([n]_q)$. Since $(\ph_m(q),[n]_q) = (1)$ (recall that $m$ does not divide $n$) there can be no nontrivial map between the modules $\{m\}_{m-c-1}$ and $R/([n]_q)$. It follows that the
differential $d_m^{n-m, m-1}$ must be zero.

As a consequence the $E_2$ part described before collapses to
$E_\infty$ and we have a copy of $\{m\}_{m-c -1}$ as a direct
summand of $H^{n-2j-1}(C_n)$ for $j= 0, \ldots, i-2$, that is for
$m \leq \frac{n}{j+1}$.

Now we consider an integer $m$ that divides $n$ and let $i =
\frac{n}{m}$. The modules of type $\{m \}[t^\pmu]$ are:

$$
\begin{array}{lll}
E_1^{\lambda m-1,n-\lambda(m-2) -2i+1}
& \mbox{generated by }& z_m(i-\lambda)01^{\lambda m-1} \mbox{ for } 1 \leq \lambda \leq i-1 \\
E_1^{\lambda m ,n-\lambda(m-2) -2i+1} & \mbox{generated by }&
v_m(i-\lambda)01^{\lambda m} \mbox{ for } 0 \leq \lambda \leq i-1.
\end{array}
$$
The situation is shown in the next diagram ($\q{h} = \{m
\}[t^\pmu]$):
\begin{center}
\begin{tabular}{|l}
\xymatrix @R=1pc @C=1pc {
\q{h}\ar[rrrd]^{d_{m-1}}\\
& & & \q{h}\ar[r]^0 & \q{h}\ar[rrrd]^{d_{m-1}} \\
& & &       &       & & & \cdots \ar[r]^0 & \cdots \ar[rrrd]^{d_{m-1}}\\
& & &       &       & & &        &        & & & \q{h} \ar[r]^0 & \q{h} \\
& & &       &       & & &        &        & & &                   & }\\
\hline
\end{tabular}
\end{center}

The map $$d_1:E_1^{\lambda m-1,n-\lambda(m-2) -2i+1} \to
E_1^{\lambda
  m, n-\lambda(m-2) -2i+1}$$ is given by the multiplication by
$\qbin{\lambda m }{\lambda m -1}'_{q,t} = [\lambda
m]_q(1+tq^{\lambda
  m-1})$, but in this case the coefficient is zero in the module $\{
m\} [t^\pmu]$ because $\ph_m(q) \mid [\lambda m]_q$ and so we have
that $E_1 = \cdots = E_{m-1}$. So we have to consider the map
$$d_{m-1}^{\lambda m, n - \lambda (m-2) - 2i +1}: E_{m-1}^{\lambda m,
  n-\lambda(m-2) -2i+1} \to E_1^{(\lambda+1) m-1,n-(\lambda+1) (m-2)
  -2i+1}$$ for $\lambda = 0, \ldots, i-2$.

This map corresponds to the multiplication by
$$\qbin{(\lambda  +1)m -1}{\lambda m }'_{q,t} = \qbin{(\lambda +1)m-1}{\lambda m}_q
\prod_{j = \lambda m +1}^{(\lambda + 1)m-1}(1+tq^{j-1}).$$ It is
easy to see that the polynomial $\qbin{(\lambda +1)m -1}{\lambda
m}_q$ is prime with the torsion $\ph_m(q)$ and so the map
$d_{m-1}^{\lambda m,
  n - \lambda (m-2) - 2i +1}$ is injective and the cokernel is
isomorphic to
$$
R \left/ \left( \ph_m(q), \prod_{j = \lambda m+1}^{(\lambda +
1)m-1}(1+tq^{j-1}) \right) \right. \simeq \bigoplus_{0 \leq k \leq m-2} \{m\}_k.
$$
As a consequence we have that
$$
\begin{array}{llcl}
E_m^{\lambda m-1,n-\lambda(m-2) -2i+1} & = &
\bigoplus_{0 \leq k \leq m-2} \{m\}_k & \mbox{ for } 1 \leq \lambda \leq i -1  \\
E_m^{\lambda m ,n-\lambda(m-2) -2i+1} & = & 0 & \mbox{ for } 0 \leq
\lambda \leq i -2.
\end{array}
$$
and all these modules collapse to $E_\infty$. This means that we
can find $\ph_m(q)$-torsion only in $H^{n-2j}(C_n)$ and for $j
\geq 1$ the summand is given by
$$
\bigoplus_{0 \leq k \leq m-2} \{m\}_k
$$
for $d \leq \frac{n}{j+1}$.

We still have to consider all the terms $E_m^{n-m,m-1} =
\{m\}[t^\pmu]$ for $m \mid n$. Here the maps we have to look at
are the following:
$$
d_m^{n-m,m-1}: E_m^{n-m,m-1} \to E_m^{n,0}.
$$
These maps correspond to multiplication by the polynomials
$\qbin{n}{n-m}'_{q,t} $. Moreover recall that $$E_1^{n,0} =
R \left/ \left( \qbin{n}{n-1}'_{q,t} \right) \right. .$$ We can now use Lemma \ref{l:ideali2} to
say that all the maps $d_m^{n-m,m-1}$ are injective and Lemma
\ref{l:ideali1} to say that
$$
E_{n+1}^{n,0} = E_\infty^{n,0} = \bigoplus_{m \mid n, 0 \leq k
\leq
  d-2} \{ m \}_k \oplus \{1 \}_{n-1}.
$$
Since $E_\infty^{n,0} = H^n(C_n)$, this complete the proof of the
Theorem.
\end{proof}

\subsection{Other computations}
We may consider the cohomology of $G_{B_n}$ over the module
$\Q[t^{\pm 1}],$ where the action is trivial for the generators
$\epsilon_1, \ldots, \epsilon_{n-1}$ and $(-t)$-multiplication for
the last generator $\overline{\epsilon}_n$. This cohomology is
computed by the complex $C_n^*$ of Section 3 where we specialize $q$
to $-1.$ So we may use similar filtration and associated spectral
sequence. We used this argument in \cite{CMS}. Here we briefly
indicate a different and more concise method, using the results of
Theorem \ref{t:cohomqt}. We have:
\begin{teo}\label{teo:ratio}
$$\begin{array}{cclc}
H^k(G_{B_n},\Q[t^{\pmu}]) & = & \Q[t^{\pmu}]/(1+t) & \quad \ 1\leq
k\leq n-1 \\
&&&\\
H^n(G_{B_n},\Q[t^{\pmu}]) & = & \Q[t^{\pmu}]/(1+t) & \quad \text{for odd $n$}\\
&&&\\
 H^n(G_{B_n},\Q[t^{\pmu}]) & = & \Q[t^{\pmu}]/(1-t^2) & \quad \text{for
even $n$.}
\end{array}$$
\end{teo}
\bigskip

\begin{proof}[Sketch of proof] Consider the short exact sequence:
$$ 0 \to \Q[q^{\pmu},t^{\pmu}] \stackrel{1+q}{\too} \Q[q^{\pmu},t^{\pmu}] \to \Q[t^{\pmu}] \to 0 $$
and the induced long exact sequence for cohomology
$$ \cdots \! \to \! H^i(G_{B_n},\! \Q[q^{\pmu}\! ,\! t^{\pmu}])\! \stackrel{1+q}{\too}
\! H^i(G_{B_n},\! \Q[q^{\pmu} \! ,\! t^{\pmu}])\! \to \! H^i(G_{B_n}, \! \Q[t^{\pmu}])\! \to
\! \cdots .$$ The result is now a straightforward consequence of
Theorem \ref{t:cohomqt}.
\end{proof}

\section{More consequences}
By means of Shapiro's lemma (see for instance \cite{brown}), the
inclusions introduced in Section \ref{sec:inclusions} can be
exploited to link the cohomology of the Artin group of type
$\tilde{A}_{n-1}$, $A_n$ to the cohomology of $\br{B_n}$.

\subsection{Cohomology of $\br{\tilde{A}_{n-1}}$}

Let $M$ be any domain and let $q$ be a unit of $M.$  We indicate by
$M_q$ the ring $M$  with the $G_{\tilde{A}_{n-1}}$-module structure
where the action of the standard generators is given by
$(-q)$-multiplication.

\begin{prop}\label{prop:sha1}
We have
\begin{align*}
H_*(G_{\tilde{A}_{n-1}}, M_q) \cong & H_*(G_{B_{n}}, M[t^{\pm
1}]_{q,t}) \\
H^*(G_{\tilde{A}_{n-1}}, M_q) \cong & H^*(G_{B_{n}}, M[[t^{\pm
1}]]_{q,t})
\end{align*}
where the action of $G_{B_n}$ on $M[t^{\pm 1}]_{q,t}$ (and on
$M[[t^{\pm 1}]]_{q,t}$) is given by $(-q)$-mul\-ti\-pli\-ca\-tion
for the generators $\epsilon_1, \ldots, \epsilon_{n-1}$ and
$(-t)$-multiplication for the last generator $\bar{\epsilon}_n$.
\end{prop}
\begin{proof}
Applying Shapiro's lemma to the inclusion
$\tilde{A}_{n-1}<G_{B_{n}}$, one obtains:
\begin{align*}
H_*(G_{\tilde{A}_{n-1}}, M_q)\cong & H_*(G_{B_{n}},
\mathrm{Ind}_{G_{\tilde{A}_{n-1}}}^{G_{B_n}}M_q) \\
H^*(G_{\tilde{A}_{n-1}}, M_q)\cong & H^*(G_{B_{n}},
\mathrm{Coind}_{G_{\tilde{A}_{n-1}}}^{G_{B_{n}}}M_q).
\end{align*}
By Corollary \ref{cor:peifer}, any element of
$\mathrm{Ind}_{G_{\tilde{A}_{n-1}}}^{G_{B_n}}M_q:=\mathbb{Z}[G_{B_n}]
\otimes_{G_{\tilde{A}_{n-1}}} M_q$ can be represented as a sum of
elements of the form $\tau^{\alpha} \otimes q^m$. Now, we have an
isomorphism of $\mathbb{Z}[G_{B_n}]$-modules
\begin{align*}
\mathbb{Z}[G_{B_n}] \otimes_{G_{\tilde{A}_{n-1}}} M_q &\to
M[t^{\pm 1}]_{q,t}
\end{align*}
defined by sending $ \tau^\alpha \otimes q^m\mapsto (-1)^{n\alpha}
t^{\alpha}q^{(n-1)\alpha +m}$ and the result follows.

In cohomology we have similarly:
\begin{align*}
\mathrm{Coind}_{G_{\tilde{A}_{n-1}}}^{G_{B_n}}M_q:=\mathrm{Hom}_{G_{\tilde{A}_{n-1}}}(\mathbb{Z}[G_{B_n}],M_q)
 \cong M[[t^{\pm 1}]]_{q,t}.
\end{align*}
\end{proof}

By Propositions \ref{prop:sha1}, in order to determine the
cohomology  $H^*(G_{\tilde{A}_{n-1}}, M_q)$ it is necessary to know
the cohomology of $\br{B_n}$ with values in the module $M[[t^{\pm
1}]]$ of Laurent series in the variable $t$. The latter is linked to
the cohomology with values in the module of Laurent polynomials by:

\begin{prop}[Degree shift]\label{prop:shift}
$$
H^*(G_{B_n}, M[[t^{\pm 1 }]]_{q,t})\cong H^{*+1}(G_{B_n}, M[t^{\pm
1 }]_{q,t}).
$$
\qed
\end{prop}
 This result  was obtained in \cite{C} in a slightly weaker form, but it is
possible to extend it to our case with little effort.

Let from now on $M=\Q[q^{\pmu}].$ In this case we have
$M[t^{\pmu}]_{q,t}= R_{q,t},$ so we obtain the cohomology of the
Artin group of affine type $\tilde{A}_{n-1}$ with $M_q-$coefficients
by means of Theorem \ref{t:cohomqt}.

In a similar way we get the rational cohomology of
$\br{\tilde{A}_{n-1}}:$

\begin{prop}\label{prop:shaq}
We have
\begin{align*}
H_*(G_{\tilde{A}_{n-1}}, \Q)\cong & H_*(G_{B_{n}}, \Q[t^{\pm 1}]) \\
H^*(G_{\tilde{A}_{n-1}}, \Q)\cong & H^*(G_{B_{n}}, \Q[[t^{\pm
1}]])
\end{align*}
where the action of $G_{B_n}$ on $\Q[t^{\pm 1}]$ (and on
$\Q[[t^{\pm 1}]]$) is trivial for the generators $\epsilon_1,
\ldots, \epsilon_{n-1}$ and $(-t)$-multiplication for the last
generator $\overline{\epsilon}_n$.
\end{prop}
To obtain the rational cohomology of $G_{\tilde{A}_{n-1}}$ we may
 apply Proposition \ref{prop:shift} together with Theorem
 \ref{teo:ratio}.

\subsection{Cohomology of $\br{A_n}$ with coefficient in the Tong-Yang-Ma representation}

The Tong-Yang-Ma representation is an $(n+1)$-dimensional
representation of the classical braid group $\br{A_n}$ discovered
in \cite{tong}. Below we just recall it, referring to
\cite{sysoeva} for a discussion of its relevance in braid group
representation theory.

\begin{df}\label{df:tong}
Let $V$
be the free $\mathbb{Q}[u^\pmu]$-module of rank $n+1$. The
Tong-Yang-Ma representation  is the representation
$$\rho:G_{A_n} \to  \mathrm{Gl}_{\mathbb{Q}[u^\pmu]}(V) $$
defined w.r.t. the basis $e_1,  \ldots, e_{n+1}$ of $V$ by:
$$\rho (\sigma_i)=\left( \begin{array}{ccccc}
I_{i-1}&&&\\
&0&1&\\
&u&0&\\
&&&I_{n-i}
\end{array} \right)$$
where $I_j$ denote the $j$-dimensional identity matrix and all
other entries are zero.
\end{df}

 Notice that the image of the pure braid group
under the Tong-Yang-Ma representation is abelian; hence this
representation factors through the \emph{extended Coxeter group}
presented in \cite{tits}.

\begin{prop}\label{prop:sha2}
We have
$$ H_*(G_{B_n}, M[t^{\pm 1}]_{q,t})\cong H_*(G_{A_n}, M_q \otimes
V)
$$
$$
H^*(G_{B_n}, M[t^{\pm 1}]_{q,t})\cong H^*(G_{A_n}, M_q \otimes V)
$$
where each generator of $G_{A_n}$ acts on $M_q \otimes V$ by
$(-q)$-multiplication on the first factor and by the Tong-Yang-Ma
representation the second factor.
\end{prop}
\begin{proof}[Sketch of proof.]
For the statement in homology, by Shapiro's lemma, it is enough to
show that $\mathrm{Ind}_{G_{B_n}}^{G_{A_n}} M[t^{\pm 1 }]_{q,t}
\cong M_q \otimes
V$.\\
 Notice that $[G_{A_n}:G_{B_n}]=n+1$ and let choose as coset
 representatives for $G_{A_n}/G_{B_n}$ the
 elements $\alpha_i= (\sigma_i \sigma_{i+1} \cdots \sigma_{n-1})
 \sigma_n (\sigma_{i} \sigma_{i+1} \cdots
 \sigma_{n-1})^{-1}$ for $1 \leq i \leq n-1$, $\alpha_n=\sigma_n$,
 $\alpha_{n+1}=e$. \\
 Then by definition of induced representation, there is an isomorphism of left $G_{A_n}$-modules,
 $$\mathrm{Ind}_{G_{B_n}}^{G_{A_n}} M[t^{\pm 1
 }]_{q,t}=\bigoplus_{i=1}^{n+1} M[t^{\pm 1}]e_i$$
where the action is on the r.h.s. is as follows.
 For an element  $x\in G_{A_n}$, write $x \alpha_k= \alpha_{k'} x'$
 with $x' \in G_{B_n}$.
 Then $x$ acts on an element $r\cdot e_k \in \bigoplus_{i=1}^{n+1}
 M[t^{\pm 1}]e_i$
 as $x(r\cdot e_k)=(x'r) \cdot e_{k'}$.\\
Computing explicitly this action for the standard generators of
$G_{A_n}$, we can write the representation in the following matrix
form:
$$ \sigma_i\mapsto \left( \begin{array}{ccccc}
-qI_{i-1}&&&\\
&0&-q&\\
&q^{-1} t&0&\\
&&&-qI_{n-i}
\end{array} \right)$$
for $1\leq i \leq n-1$, whereas
$$ \sigma_n\mapsto \left( \begin{array}{ccccc}
-qI_{n-1}&&\\
&0&1&\\
& -t&0
\end{array} \right).$$
Conjugating by $U=\mathrm{Diag}(1,1, \ldots, 1,-q^{-1})$ and
setting $u=-q^{-2}t$, one obtains the desired result.\\
Finally, since $[G_{A_n}:G_{B_n}]=n+1<\infty$, the induced and
coinduced representation are isomorphic; so the analogous
statement in cohomology follows.
\end{proof}

In particular the cohomology of $\br{B_n}$ determined in Theorem
\ref{t:cohomqt} is isomorphic to the cohomology of $G_{A_n}$ with
coefficient in the Tong-Yang-Ma representation twisted by an
abelian representation. \vs

 By means of Shapiro's lemma, we may as
well determine the cohomology of $\br{A_n}$ with coefficient in
the Tong-Yang-Ma representation. Indeed:

\begin{prop}\label{prop:shaq2}
We have
\begin{align*}
    H_*(G_{B_n}, \Q[t^{\pm 1}])&\cong H_*(G_{A_n}, V)\\
    H^*(G_{B_n}, \Q[t^{\pm 1}])&\cong H^*(G_{A_n}, V)
\end{align*}
where  $V$ is the representation of $G_{A_n}$ defined in
\ref{df:tong}.
\end{prop}

As a consequence we have

\begin{cor} Let $V$ be the $(n+1)$-dimensional representation of the
braid group $\mathrm{Br}_{n+1}$ defined in \ref{df:tong}. Then the
cohomology
$$H^*(\mathrm{Br}_{n+1},\ V)$$
is given as in Theorem \ref{teo:ratio}.
\end{cor}

\begin{os}
In particular the homology of $G_{\tilde{A}_{n-1}}$ with trivial
coefficients is isomorphic to the homology of $G_{A_n}$ with
coefficients in the Tong-Yang-Ma representation.
\end{os}
\section{Related topological constructions}
We refer to \cite{CMS} for the few changes  which have to be done to
the construction given in \cite{S2} (see also \cite{S1}) for
non-finite type Artin groups (but still finitely generated). We
obtain a {\it finite} CW-complex $X_W$, explicitly described,  which
is a deformation retract of the {\it orbit space} of the Artin
group. The latter is defined as the quotient space

$$\mathbf{M}({\mathcal A})_W:= \mathbf{M}({\mathcal A})/W.$$
where

$$\mathbf{M}({\mathcal A})\ :=\ [U^0\ +\ i\R^n]\ \setminus\ \bigcup_{H\in {\mathcal
   A}}\ H_{\C}$$
$U^0\subset \R^n$ being the interior part of the {\it Tits cone} of
$W,$ while $\mathcal A$ is the hyperplane arrangement of $W.$ The
associated Artin group $G_W$ is the fundamental group of the orbit
space (see \cite{bour,vin, Br, Ng, van}).

The simplest way to realize $X_W$ is by taking one point $x_0$
inside a {\it chamber} $C_0$ and, for any maximal subset $J\subset
S$ such that the parabolic subgroup $W_J$ is finite, construct a
$|J|$-cell (a polyhedron) in $U^0$ as the ``convex hull'' of the
$W_J$-orbit of $x_0$ in $\R^n.$ So, we obtain a finite cell complex
which is the union of (in general, different dimensional) polyhedra.
Next, there are identifications on the faces of these polyhedra,
which are the same as described in \cite{S2} for the finite case.
The resulting quotient space is a CW-complex $X_W$  which has a
$|J|$-cell for each $J\subset S$ such that $W_J$ is finite. We show
an example in the case $\tilde{A}_2$ (fig. \ref{fig:affine}).

\begin{figure}[hbtp]
\begin{center}
\scalebox{.8}{\input{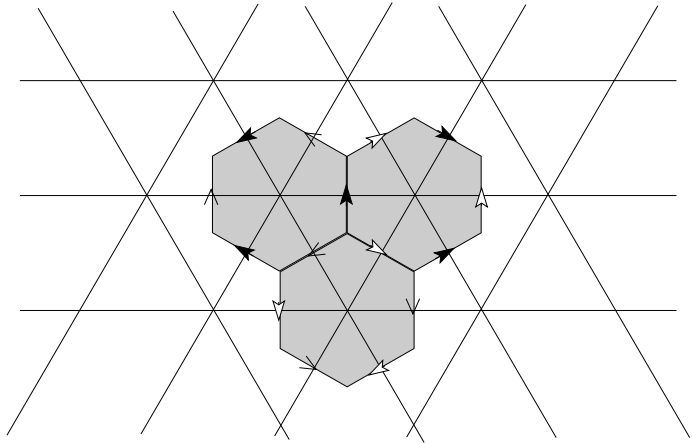}} \caption{{\it the
space} $K(G_{\tilde{A_2}},1)$ {\it is given as union of $3$
  exagons with edges glued according to the arrows (there are: $1$ 0-cell, $3$
  1-cells, $3$ 2-cells in the quotient).}}
\label{fig:affine}
\end{center}
\end{figure}

\begin{os} When $W$ is an affine group, the orbit space is known
to be a $K(\pi,1)$ for types $\tilde{A}_n,\ \tilde{C}_n$ (see
\cite{Ok,CP}) and recently for type $\tilde{B}_n$ (\cite{CMS1}); see
\cite{CD} for further classes.
\end{os}

\begin{os} The standard presentation for $G_W$ is quite easy to
  derive from the topological description of $X_W$; we may thus recover Van del Lek's result \cite{van}.
\end{os}

It follows

\begin{prop} Let $K_W^{fin}:=\{J\subset S\ :\ |W_J|<\infty\}$ with the
  natural structure of simplicial complex. Then the Euler
  characteristic of the orbit space (so, of the
  group $G_W$ when such space is of type $K(\pi,1)$)) equals
$$\chi(K_W^{fin}). $$
If \ $W$\  is affine of rank \ $n+1$\  we have
$$\chi(\mathbf{M}({\mathcal A})_W)\ =\ \chi(K_W^{fin})=\  1-\chi(S^{n-1})\ =\ (-1)^n $$
If $W$ is $\mathrm {two-dimensional}$ (so, all 3-subsets of $S$
generate an infinite group) of rank $n$ then
$$\chi(\mathbf{M}({\mathcal A})_W)\ =  1-n+ m$$
where $ m$ is the number of pairs in $S$ having finite weight
($m=\frac{n(n-1)}{2}$ if there are no $\infty$-edges in the Coxeter
graph).
\end{prop}
\begin{proof} The first two statements were already remarked in
\cite{CMS}. The last one is clear.
\end{proof}

\begin{os} The cohomology of the orbit space in case
  $\tilde{A}_n$ with trivial coefficients is deduced from Corollary
\ref{prop:shaq} and from Theorem  \ref{teo:ratio}; that with local
coefficients in the $G_{\tilde{A}_{n}}$-module $\Q[q^\pmu]$ is
deduced from Theorem \ref{t:cohomqt}.
\end{os}

\def\cprime{$'$} \def\cftil#1{\ifmmode\setbox7\hbox{$\accent"5E#1$}\else
  \setbox7\hbox{\accent"5E#1}\penalty 10000\relax\fi\raise 1\ht7
  \hbox{\lower1.15ex\hbox to 1\wd7{\hss\accent"7E\hss}}\penalty 10000
  \hskip-1\wd7\penalty 10000\box7}
\providecommand{\bysame}{\leavevmode\hbox to3em{\hrulefill}\thinspace}
\providecommand{\MR}{\relax\ifhmode\unskip\space\fi MR }
\providecommand{\MRhref}[2]{%
  \href{http://www.ams.org/mathscinet-getitem?mr=#1}{#2}
}
\providecommand{\href}[2]{#2}


\end{document}